\documentclass[12pt,twoside,reqno]{amsart}
\RequirePackage[english]{babel}

\usepackage[OT1]{fontenc}
\usepackage{type1cm}

\usepackage{enumerate}

\usepackage{amsthm,mathtools}
\RequirePackage{amsmath,amsfonts,amssymb,amsthm}

\RequirePackage[dvips]{graphicx}

\usepackage{psfrag}
\usepackage{verbatim}
\usepackage[usenames]{color}
\usepackage[dvips]{graphicx} % for gfx

  \numberwithin{equation}{section}

\usepackage[active]{srcltx}  % for the inverse search

  \newtheorem{theorem}{Theorem}[section]
  \newtheorem{lemma}[theorem]{Lemma}
  \newtheorem{proposition}[theorem]{Proposition}
  \newtheorem{corollary}[theorem]{Corollary}

  \theoremstyle{remark}

\DeclareSymbolFont{bbold}{U}{bbold}{m}{n}
\DeclareSymbolFontAlphabet{\mathbbold}{bbold}

\begin{document}

\title[Hausdorff dimensions of the recurrent sets]{Zero-one law of Hausdorff dimensions of the recurrent sets} 

\author{Dong Han Kim}

\address{$^1$Department of Mathematics Education, Dongguk University-Seoul, Seoul, 100-715 Korea}

\email{kim2010@dongguk.edu}

\thanks{
$^*$ Corresponding author.\\
This work was supported by the National Research Foundation of Korea(NRF) (2012R1A1A2004473), and by NSFC (no. 11201155 and 11371148).}

\author{Bing Li$^*$}

\address{$^2$Department of Mathematics, South China University of Technology,
Guangzhou, 510641, P.R. China}

\email{scbingli@scut.edu.cn}

\begin{abstract}
Let $(\Sigma, \sigma)$ be the one-sided shift space with $m$ symbols and $R_n(x)$ be the first return time of $x\in\Sigma$ to the $n$-th cylinder containing $x$.
 Denote $$E^\varphi_{\alpha,\beta}=\left\{x\in\Sigma: \liminf_{n\to\infty}\frac{\log R_n(x)}{\varphi(n)}=\alpha,\ \limsup_{n\to\infty}\frac{\log R_n(x)}{\varphi(n)}=\beta\right\},$$
where $\varphi: \mathbb{N}\to \mathbb{R}^+$ is a monotonically increasing function and $0\leq\alpha\leq\beta\leq +\infty$.
We show that the Hausdorff dimension of the set $E^\varphi_{\alpha,\beta}$ admits a dichotomy: it is either zero or one depending on $\varphi, \alpha$ and $\beta$.
\end{abstract}

\maketitle



\section{Introduction}
Let $m\geq 2$ be an integer and $(\Sigma, \sigma)$ be one-sided shift space with $m$ symbols, more precisely, $\Sigma=\{0,1,\cdots,m-1\}^\mathbb{N}$ and $\sigma(x)=(x_{i+1})_{i=1}^\infty$ for $x=(x_i)_{i=1}^\infty\in\Sigma$. A usual metric $d$ on $\Sigma$ is given as
$$d(x,y)=m^{-\inf\{k\geq 0: x_{k+1}\neq y_{k+1}\}}$$
for $x=(x_i)_{i=1}^\infty, y=(y_i)_{i=1}^\infty\in\Sigma.$
For $n\ge 1$ and $x\in\Sigma$, define the first return time of $x$ to the initial word of length $n$ as
\begin{align*}
R_n (x) &=\inf\{j \geq 1 : x_{j+1}x_{j+2}\cdots x_{j+n}=x_1x_2\cdots x_n\},\\
R'_n (x) &=\inf\{j \geq n: x_{j+1}x_{j+2}\cdots x_{j+n}=x_1x_2\cdots x_n\}.
\end{align*}
%That is $R_n(x) =\inf\{j\geq 1: \sigma^j (x) \in B(x, m^{-n}) \}$, where $B(x, r)$ is the closed ball centered at $x$ with radius $r$.
That is, $$R_n(x)=\inf\{j \geq 1: \sigma^j(x)|_n=x|_n\}=\inf\{j \geq 1: d(\sigma^j(x), x) \le m^{-n}\},$$ where $x|_n$ is the prefix of $x$ with the length $n$.
Both $R_n$ and $R'_n$ are the same except for the case of very short return time.
However, $R'_n$ is slightly easier to investigate.

Let $\nu$ be any $\sigma$-invariant ergodic Borel probability measure on $\Sigma$.
Ornstein and Weiss \cite{OW} proved that for $\nu$-almost all $x\in\Sigma$,
$$\lim_{n\to\infty}\frac{\log R'_n(x)}{n}=h_\nu(\sigma),$$
where $h_\nu(\sigma)$ denotes the measure-theoretic entropy of $\nu$ with respect to $\sigma$ (see also \cite{GKP06}).

The topic of the first return time of a point in a dynamical system is originated in the famous Poincar\'{e} recurrence theorem, (see \cite[p. 61]{Fu})     
which states that $\mu$-almost all $x\in X$ is recurrent in the sense  
\begin{equation}\label{Poincare}
\liminf_{n\to\infty}d(T^nx,x)=0,
\end{equation}
where $(X, \mathcal{B},\mu,T,d)$ is a metric measure-preserving dynamical system, by which we mean that $(X, d)$ is a metric space and has a countable base, $\mathcal{B}$ is a sigma-field containing the Borel sigma-field of $X$ and $(X, \mathcal{B},\mu, T)$ is a measure-preserving dynamical system.
Boshernitzan \cite{Bo} improved \eqref{Poincare} by a quantitative version
$$\liminf_{n\to\infty}n^{1/\alpha}d(T^nx,x)<\infty$$
for $\mu$-almost every (a.e.) $x\in X$, where $\alpha$ is the dimension of the space in some sense.
Let
$$\tau_r(x)=\inf\{n\geq 1: d(T^nx, x)<r\}$$
and define the lower and upper recurrence rates of $x\in X$ as
$$\underline{R}(x)=\liminf_{r\to 0}R_r(x),\ \ \overline{R}(x)=\limsup_{r\to 0}R_r(x),$$
where $R_r(x)=\frac{\log\tau_r(x)}{-\log r}$. Barreira and Saussol \cite{BS1} proved that 
\begin{equation}\label{upperlowerrate}
\underline{R}(x)=\underline{d}_\mu(x),\ \ \overline{R}(x)=\overline{d}_\mu(x), \ \ \ \mu-\text{a.e.}
\end{equation}
with the condition that $\mu$ has a so-called long return time (see \cite{BS1}) and $\underline{d}_\mu(x)>0$ for $\mu$-a.e. $x$, where $\underline{d}_\mu(x)$, $\overline{d}_\mu(x)$ are the lower and upper pointwise dimensions of $\mu$ at $x\in X$ respectively.  
As a consequence of \eqref{upperlowerrate}, if $\alpha>\underline{d}_\mu(x)$, then
$$\liminf_{n\to\infty}n^{1/\alpha}d(T^nx, x)=0,$$
which is a reformulation of Boshernitzan's result. Some other sufficient conditions of  \eqref{upperlowerrate} were given in the literatures (for instance, \cite{Rousseau12, RousseauSaussol10, Saussol06}). 
The distribution of the first hitting time of the dynamical system is also  considered (see \cite{Galatolo07, GPM10}).

It is shown by Tan and Wang \cite[Theorem 1.3]{TW11} (see also \cite{HV95}) that   
for a positive function $\psi$ defined on $\mathbb{N}$ we have
\begin{equation}\label{dimTW11}
\dim_{\rm H}\left\{x\in\Sigma: d(\sigma^n(x),x)<\psi(n)\ \ \text{infinitely often}\right\}=\frac{1}{1+b},
\end{equation}
where % $$b= \max\left\{\liminf\limits_{n\to\infty}\frac{-\log_m\psi(n)}{n}, 0\right\}.$$
$$b= \liminf\limits_{n\to\infty}\frac{-\log_m  \min\left\{ \psi(n), 1 \right\} }{n}.$$   

In 2001, Feng and Wu \cite{FW} studied the exceptional sets
$$\left\{x\in\Sigma: \liminf_{n\to\infty}\frac{\log R'_n(x)}{n}=\alpha,\ \limsup_{n\to\infty}\frac{\log R'_n(x)}{n}=\beta\right\}$$
with $0\le\alpha\le\beta\le\infty$ and proved that those sets are always of Hausdorff dimension one no matter what $\alpha$ and $\beta$ are (see also \cite{CWY15, PTW12, SW}). 
Lau and Shu \cite{LS} extended this result to the dynamical systems with specification property by considering the topological entropy instead of Hausdorff dimension.  
Olsen \cite{Olsen} studied  the set of the points for which the set of accumulation points of
$\frac{\log R'_n(x)}{n}$ is a given interval for the self-conformal system satisfying a certain separation condition  (see also \cite{Olsen04}).  
He proved such set is of full Hausdorff dimension which can be applied to the case of the $N$-adic transformation with $N\in\mathbb{N}$. 
Ban and Li \cite{BanLi2014} generalised this result to $\beta$-transformation for any $\beta>1$ including the cases of full shifts, subshift of finite type, specification,  synchronizing etc. 

In \cite{FW} and \cite{OW}, the authors considered the recurrence time $R'_n(x)$ with the exponential rate.
What will happen if we replace the exponential rate with the polynomial rate?
Denote $$E'_{\alpha,\beta}=\left\{x\in\Sigma: \liminf_{n\to\infty}\frac{\log R'_n(x)}{\log n}=\alpha,\ \limsup_{n\to\infty}\frac{\log R'_n(x)}{\log n}=\beta\right\}. $$
Peng \cite{Peng} proved that $\dim_{\rm H}E'_{\alpha,\beta}=1$ for any $1\leq\alpha\leq \beta\leq +\infty$, where $\dim_{\rm H}$ means the Hausdorff dimension of some set.

However, for the short return time case it is worth to distinguish $R_n$ and $R'_n$.
Let $0\le\alpha\le\beta\le\infty$ and denote $$E_{\alpha,\beta}=\left\{x\in\Sigma: \liminf_{n\to\infty}\frac{\log R_n(x)}{\log n}=\alpha,\ \limsup_{n\to\infty}\frac{\log R_n(x)}{\log n}=\beta\right\}. $$
Once the condition $\alpha\geq 1$ fails for $R_n$, 
we have $\dim_{\rm H}E_{\alpha,\beta}=0$ (see Corollary \ref{lograte}), which is the complement of the result in \cite{Peng}.
This result indicates that the Hausdorff dimensions of the recurrence sets with the polynomial rates may drop from the full dimension which is different with that in \cite{FW}. Furthermore, such dimensions are either zero or one.

Now we consider the case of general rate. More precisely, let $\varphi: \mathbb{N}\to\mathbb{R}^+$ be a monotonically increasing function.
 Denote $$E^\varphi_{\alpha,\beta}=\left\{x\in\Sigma: \liminf_{n\to\infty}\frac{\log R_n(x)}{\varphi(n)}=\alpha,\ \limsup_{n\to\infty}\frac{\log R_n(x)}{\varphi(n)}=\beta\right\}.$$
We completely calculate the Hausdorff dimensions of the sets $E^\varphi_{\alpha,\beta}$ by the following theorem, which is a generalization and complement of \cite{FW} and \cite{Peng}.

Write
$$\liminf_{n\to\infty}\frac{\varphi(n)}{\log n}=\delta\ \ \text{and}\ \ \limsup_{n\to\infty}\frac{\varphi(n)}{\log n}=\gamma.$$
\begin{theorem}\label{generalrate}
For any $0 \le \alpha \le \beta \le \infty$, we have
$$
\dim_{\rm H} (E^\varphi_{\alpha,\beta}) = \begin{dcases}
1, &\text{ if } \ \alpha\geq\frac{1}{\gamma} \text{ and }
\beta\geq\frac{1}{\delta}, \\
0, &\text{ otherwise}, \end{dcases}
$$
with the convention $\frac{1}{0}=\infty$ and $\frac{1}{\infty}=0$.
\end{theorem}

Application of Theorem~\ref{generalrate} to $\varphi(n)=\log n$ implies the following.
\begin{corollary}\label{lograte}
Let $0\le\alpha\le\beta\le\infty$. Then
\begin{equation*}
\dim_{\rm H}E_{\alpha,\beta}=
\begin{cases}
1& \text{if}\ \ \alpha\ge 1,\\
0& \text{if}\ \ \alpha<1.
\end{cases}
\end{equation*}
\end{corollary}

From the results in \cite{FW,Peng} and Theorem \ref{lograte}, when $\varphi(n)=n$ or $\varphi(n)=\log n$, the values of $\dim_{\rm H}(E^\varphi_{\alpha,\beta})$ just depend on $\alpha$ ($\beta\geq\alpha$ is a natural condition since otherwise, the set $(E^\varphi_{\alpha,\beta})$ will be empty). 
The new phenomenon arises for general $\varphi$ from Theorem \ref{generalrate}, that is, the values of $\dim_{\rm H}(E^\varphi_{\alpha,\beta})$ may depend on $\beta$ as well.

%By Tan and Wang's work \eqref{dimTW11}, we have
%$$\dim_H \left \{x\in\Sigma: \liminf_{n\to\infty}\frac{R_n(x)}{n} \le \alpha \right \} = \frac{\alpha}{\alpha +1}.$$

%\begin{theorem}[conjecture]
%$$\dim_H \left \{x\in\Sigma: \limsup_{n\to\infty}\frac{R_n(x)}{n} =  \beta \right \} = \frac{\beta -1}{\beta+1}.$$
%For $\alpha + 1 < \beta$, we have
%$$\dim_H \left \{x\in\Sigma: \liminf_{n\to\infty}\frac{R_n(x)}{n} =  \alpha, \quad \limsup_{n\to\infty}\frac{R_n(x)}{n} =  \beta \right \} = 1 - \frac{\alpha+ 1}{\beta}.$$
%\end{theorem}

%Small idea: $R_n$ is constant for $n_i < n \le n_i$. Then 
%$$ \dim = \liminf \frac{R_{n_{i+1}} - (R_{n_i} + n_i) + R_{n_{i}} - (R_{n_{i-1}} + n_{i-1}) + \dots  }{n_{i+1} + R_{n_{i+1}}} $$

\section{Proof of the main result}
%Firstly we prove Theorem \ref{lograte} by \eqref{dimTW11} and Theorem 1 in \cite{Peng}.
%As a consequence of \eqref{dimTW11}, we have the following proposition which will be used to prove Theorem \ref{lograte}.

Firstly we prove the zero-dimensional part of Theorem \ref{generalrate} by the following proposition, which is
a consequence of \eqref{dimTW11}.
\begin{proposition}\label{twodimensions}
Let $\alpha <1$. Then
\begin{equation}\label{lessthanalpha}
\dim_{\rm H}\left\{x\in\Sigma: \liminf_{n\to\infty}\frac{\log R_n(x)}{\log n}\leq\alpha\right\}=0.
\end{equation}
%and
%\begin{equation}\label{equalalpha}
%\dim_{\rm H}\left\{x\in\Sigma: \liminf_{n\to\infty}\frac{\log R_n(x)}{\log n}=\alpha\right\}=0.
%\end{equation}
\end{proposition}
\begin{proof}
For any $0<\varepsilon<1-\alpha$, the condition 
$$\liminf\limits_{n\to\infty}\frac{\log R_n(x)}{\log n}\leq \alpha$$
%\text{or}\ \liminf\limits_{n\to\infty}\frac{\log R_n(x)}{\log n}=\alpha
implies there exist infinitely many $n$'s such that $R_n(x)<n^{\alpha+\varepsilon}$, that is, $n>R_n(x)^{1/(\alpha+\varepsilon)}$.
By the definition of $R_n(x)$, we obtain $d(\sigma^{R_n(x)} (x), x)<m^{-n}<m^{-R_n(x)^{1/(\alpha+\varepsilon)}}$.
Therefore the desired sets in \eqref{lessthanalpha}  are the subsets of the set
 $\left\{x\in\Sigma: d(\sigma^n(x),x)<m^{-n^{1/(\alpha+\varepsilon)}}\ \text{infinitely often}\right\}.$
We complete the proof by the application of \eqref{dimTW11}.
\end{proof}

\begin{proof}[Proof of the zero-dimensional part of Theorem \ref{generalrate}]
Suppose that $\alpha \gamma <1$  with $\alpha<\infty$ and $\gamma <\infty$ or $\beta \delta < 1$ with $\beta < \infty$ and $\delta <\infty$.
Since
\begin{align*}
\liminf_{n\to\infty}\frac{\log R_n(x)}{\log n} 
&\le \left( \liminf_{n\to\infty}\frac{\log R_n(x)}{\varphi(n)}\right) \cdot \left( \limsup_{n\to\infty} \frac{\varphi(n)}{\log n}\right) = \alpha \gamma, \\
\liminf_{n \to \infty}\frac{\log R_{n}(x)}{\log n} &\le 
\left( \limsup_{n \to \infty}\frac{\log R_{n}(x)}{\varphi(n)} \right) \cdot
\left( \liminf_{n \to \infty}\frac{\varphi(n)}{\log n} \right) = \beta \delta,
\end{align*}
we have
$$ \liminf_{n\to\infty}\frac{\log R_n(x)}{\log n} < 1.$$
By Proposition \ref{twodimensions},  such point set is zero dimensional, which implies that
\begin{equation*}
\dim_{\rm H} (E^\varphi_{\alpha,\beta}) = 0. \qedhere
\end{equation*}

\end{proof}

Now we  concentrate on the proof of the one-dimensional part of Theorem \ref{generalrate}. The idea is to construct a subset of $E_{\alpha, \beta}^\varphi$ with full Hausdorff dimension for different cases of $\varphi, \alpha, \beta$.  The following technical lemma provides such subsets with dimension one, which is a generalization of Lemma 1 in \cite{FW} and the ideas of proofs are same.

\begin{lemma}\label{Cantorsetlemma}
Let $\{n_i\}_{i\geq 1}$ and $\{\ell_i\}_{i\geq 1}$ be two strictly increasing sequences of natural numbers satisfying the following conditions:\\
(i) $\ell_{i+1}\geq \ell_i+n_i+3$;\\
(ii) $\lim\limits_{i\to\infty}\frac{i(n_i+3)}{\ell_i}=0$.\\
 Then the set
$$A(\{n_i\}, \{ \ell_i \}):= \left\{x\in\Sigma: R_n(x)=\ell_{i+1} \ \text{for all large $n$ with } n_i< n\le  n_{i+1}\right\}  
%\left\{x\in\Sigma: \exists i_0\in\mathbb{N}\ \text{s.t.}\  R_n(x)=\ell_{i+1}\ \forall i\geq i_0,  n_i<\forall n\le  n_{i+1}\right\}
$$
is of full Hausdorff dimension.
\end{lemma}

\begin{proof}
Let $p>2$ be a natural number. Define
$$F_p=\left\{x\in\Sigma: x_j=0\ \text{for}\ 1\le j\le p, x_{pk+1}=x_{pk+p}=1\ \text{for}\ k\ge 1\right\}.$$
By the dimensional formula of the self-similar set (see \cite[p.130]{Fal90}, 
see also \cite{FW}), we know that
$$\dim_{\rm H}F_p=\frac{p-2}{p}.$$

We will construct a map $g: F_p\to A(\{n_i\}, \{ \ell_i \})$ in the following and show that $g^{-1}$ is nearly Lipschitz on $g(F_p)$,      
that is, for any $\varepsilon>0$, there exists $M>0$ such that $d(g(x),g(y))<2^{-k}$ implies $d(x,y)<2^{-(1-\varepsilon)k}$ for all $k\ge M$.

Let $x\in F_p$ and $x^{(0)}=x$.
 Find the minimal $k_0$  such that $\ell_{k_0}-1>p$ and let $x^{(1)}=x^{(2)}=\cdots=x^{(k_0-1)}=x^{(0)}$.
%Denote the word $w_1=1(x^{(0)}|_{n_1})\overline{x^{(0)}_{n_1+1}}1$, where $x^{(0)}|_{n_1}$ means the prefix word of $x^{(0)}$ with length $n_1$ and $\overline{x^{(0)}_{n_1+1}}$ is the symbol different with $x_{n_1+1}$. We obtain $x^{(1)}$ by inserting $w_1$ in $x^{(0)}$ at the place $\ell_{1}$, that is,
%$$x^{(1)}=x_1^{(0)}x_2^{(0)}\cdots x_{\ell_1}^{(0)}w_1x_{\ell_1+1}^{(0)}x_{\ell_1+2}^{(0)}\cdots.$$
Assume that the sequence $x^{(k-1)}=(x_i^{(k-1)})$ has been defined    
($k\geq k_0$) and now we define $x^{(k)}$. Put the word $w_{k}(x)=1(x^{(k-1)}|_{n_{k}})\overline{x^{(k-1)}_{n_{k}+1}}1$, where $x^{(k-1)}|_{n_k}$ means the prefix word of $x^{(k-1)}$ with length $n_k$ and $\overline{x^{(k-1)}_{n_k+1}}=x^{(k-1)}_{n_k+1} +1 \pmod m.$ 
%$\overline{x^{(k-1)}_{n_k+1}}$ is a symbol different from $x_{n_k+1}^{(k-1)}$.
 Then $x^{(k)}$ is given by inserting $w_k(x)$ in $x^{(k-1)}$ at the place $\ell_{k}$, i.e.,
$$x^{(k)}=x_1^{(k-1)}x_2^{(k-1)}\cdots x_{\ell_k-1}^{(k-1)}w_k(x)x_{\ell_k}^{(k-1)}x_{\ell_k+1}^{(k-1)}\cdots.$$
Since $\{\ell_k\}$ is increasing to infinity and $x^{(k)}|_{\ell_k-1}=x^{(k-1)}|_{\ell_k-1}$, the limit of $\{x^{(k)}\}$ exists, denoted by $x^*$. 
Indeed, 
$x^*$ is obtained by inserting the sequence of words $\{w_k(x)\}$ in $x=x^{(0)}_1x^{(0)}_2\cdots x^{(0)}_n\cdots.$

Now we prove $x^*\in A(\{n_i\}, \{ \ell_i \})$.  
In fact, let $n_i<n\le n_{i+1}$ for some $i\geq 1$ with $n_i>p$. 
Since $x^*|_{\ell_{i+1}+n_{i+1}+3}=x^{(i+1)}|_{\ell_{i+1}+n_{i+1}+3}$ (by condition (i)), by the construction of $x^{(i+1)}$, we know that $R_n(x^*)=R_n(x^{(i+1)})\leq \ell_{i+1}$. 
It remains to show that $R_n(x^{(i+1)})< \ell_{i+1}$ can not happen, that is, $x^{(i+1)}|_n$  cannot reappear in $x^{(i+1)}|_{\ell_{i+1}+n-1}=(x^{(i)}|_{\ell_{i+1}-1})1(x^{(i)}|_{n-1})$. 
In fact, since $x^{(i+1)}|_n$ begins with $p$ consecutive zeros which does not appear in $x$ except at the beginning by the structure of $x\in F_p$, the only possible places where $x^{(i+1)}|_n$ may appear are $w_j(x)$ for some $1\le j\le i$. 
However,  $x^{(i+1)}|_n$ cannot appear in $w_j(x)$ for all $1\le j\le i$ because the maximal length of common prefixes between $x^{(i+1)}|_n$ and $w_j(x)$ is $n_j$ due to the existence of $\overline{x_{n_j+1}^{(j-1)}}$ in $w_j(x)$ and $n>n_j$.  
Therefore, $R_n(x^*)= \ell_{i+1}$, that is, $x^*\in A(\{n_i\}, \{ \ell_i \})$. 
So we have the map $g: F_p\to A(\{n_i\}, \{ \ell_i \})$ defined as $g(x)=x^*$. 

In the following, we show that $g^{-1}$ is nearly Lipschitz on $g(F_p)$. 
Indeed, suppose $d(x^*, y^*)< m^{-k}$ for some $x, y \in F_p$ 
 and $\ell_i\leq k < \ell_{i+1}$, then $x_1^*x_2^*\cdots x_k^*=y_1^*y_2^*\cdots y_k^*.$ 
Since $x=g^{-1}(x^*)$ is obtained by removing the parts of $w_j(x)$ from $x^*$ and similar for $y$, it turns out that $x_1x_2\cdots x_{k'}=y_1y_2\cdots y_{k'}$, where
$$k  =  k-\sum_{j=1}^{i}(n_j+3).$$   
Since $\sum_{j=1}^i(n_j+3)\leq i(n_i+3)$, by condition (ii) and $\ell_i\leq k$, for any $\varepsilon>0$, there exists $M>0$ such that for all $k>M$, we have 
$$\sum_{j=1}^i(n_j+3)\leq i(n_i+3)<\ell_i\varepsilon <\varepsilon k.$$   
So $d(x,y)\leq m^{-k'} <m^{-(1-\varepsilon)k}.$ 
Thus
$$d(g^{-1}(x^*),g^{-1}(y^*))\leq d(x^*,y^*)^{1-\varepsilon}.$$

Therefore, $\dim_{\rm H}g(F_p)\geq (1-\varepsilon)\dim_{\rm H}F_p$  (see \cite[Proposition 2.3]{Fal90}).
So $\dim_{\rm H}g(F_p)\geq \dim_{\rm H}F_p$ by letting $\varepsilon\to 0$. 
Notice that $A(\{n_i\}, \{ \ell_i \})\supset g(F_p)$ and $\dim_{\rm H}F_p = \frac{p-2}{p}$,   
we have
$$\dim_{\rm H}A(\{n_i\}, \{ \ell_i \})\geq \frac{p-2}{p}.$$    
We conclude the assertion of the lemma by letting $p\to\infty$. 
\end{proof}

%In the following we prove Theorem \ref{generalrate}. The zero dimensional part can easily get under the help of Proposition \ref{twodimensions}. 
The proof  of the one-dimensional part relies on the applications of Lemma \ref{Cantorsetlemma} by constructing proper sequences $\{n_i\}$ and $\{\ell_i\}$, and verifying the corresponding $A(\{n_i\}, \{ \ell_i \})$ is a subset of $E_{\alpha, \beta}^\varphi$.
For any $x\in A(\{n_i\}, \{ \ell_i \})$, since $R_n(x)=\ell_i$ for all $n_{i-1}+1\leq n\leq n_i$, we know that 
 \begin{align*}
\liminf_{n\to\infty}\frac{\log R_n(x)}{\varphi(n)}&=\liminf_{i\to\infty}\frac{\log\ell_i}{\varphi(n_i)},\\  
\limsup_{n\to\infty}\frac{\log R_n(x)}{\varphi(n)}&=\limsup_{i\to\infty}\frac{\log\ell_i}{\varphi(n_{i-1}+1)}.
\end{align*}
So the subsequences $\{n_i\}$ and $\{n_i+1\}$ are essential for the lower and upper recurrence rates.
The following two technical 
lemmas are needed for the construction of the sequences $\{n_i\}$ and $\{\ell_i\}$.

\begin{lemma}\label{subseq1}
For a positive 
 function $\varphi(n)$ and a constant $C \ge 1$ we can choose a sequence of positive integers $\{n_i\}_{i =1}^\infty$ satisfying that
\begin{equation*}
\limsup_{i \to \infty} \frac{\varphi(n_i)}{\log n_i} = \gamma, \qquad
\liminf_{i \to \infty} \frac{\varphi(n_i+1)}{\log (n_i+1)} = \delta
\end{equation*}
and
$$ \lim_{i \to \infty} \dfrac{\log n_{i+1} }{\log n_i} = C, \qquad \lim_{i \to \infty} \frac{\log n_i}{i} = \infty. $$
\end{lemma}

\begin{proof}
%Let
%$$\limsup_{n \to \infty} \frac{\varphi(n)}{\log n} = \gamma, \quad \liminf_{n \to \infty} \frac{\varphi(n)}{\log n} = \delta.$$
First, suppose that $C >1$.
Choose $n_1 > e$.
For each $n_k$ choose an integer $m_k$ such that
$$\log (\log m_k) - \log (\log n_k) \ge k \log C $$
and
\begin{equation*}
\left| \frac{\varphi(m_k)}{\log m_k} - \gamma \right| < \frac 1k \text{ for } 0 \le  \gamma < \infty, \qquad
\frac{\varphi(m_k)}{\log m_k}  > k  \text{ for } \gamma = \infty.
\end{equation*}
Denote
\begin{equation}\label{dd}
d_k = \left \lfloor \frac{\log (\log m_k) - \log (\log n_k)}{\log C} \right\rfloor  \ge k
\end{equation}
where $\lfloor \cdot \rfloor$ is the floor function.
For $0 \le j \le d_k$ define $r_{k+j}$ by letting   
\begin{equation}\label{dx}
\log (\log r_{k + j}) = \log (\log n_k) +  \frac j{d_k}  \cdot \left( \log (\log m_k) - \log (\log n_k) \right) .
\end{equation}
Then  for $1 \le j \le d_k$  
$$
\log \left( \frac{\log r_{k + j}}{\log r_{k + j-1}} \right)  = \frac {\log (\log m_k) - \log (\log n_k)}{d_k} .   
$$
Since
\begin{equation*}
\log C \le \frac {\log (\log m_k) - \log (\log n_k) }{d_k}
< \frac {(d_k+1)\log C}{d_k} \le \left( 1 + \frac{1}{k} \right) \log C,
\end{equation*}
we have for $ 1 \le j \le d_k$
\begin{equation}\label{pc}
C \le \frac{\log r_{k + j}}{\log r_{k + j-1}} < C^{1+1/k}.
\end{equation}
Let $n_{k+j} = \lceil r_{k+j} \rceil$ for $ 1 \le j \le d_k$,
where $\lceil t\rceil$ is the ceiling function. %smallest integer not less than $t$.
Note that $n_{k+d_k} = m_k$.  

Next, we choose $m'_k$ such that
$$
\log (\log m'_k) - \log (\log n_{k+d_k}) \ge (k+d_k) \log C
$$
and
\begin{align*}
\left| \frac{\varphi(m'_k+1)}{\log (m_k'+1)} - \delta \right| < \frac 1{k+d_k}, \quad &\text{ for } 0 \le \delta < \infty,\\
%\frac{\varphi(m'+1)}{\log (m'+1)}  < - (k+d), \quad &\text{ for } \delta = - \infty, \\
\frac{\varphi(m'_k+1)}{\log (m'_k+1)}  > k+d_k, \quad &\text{ for } \delta = \infty.
\end{align*}
Then, denote $d'_k$ the same way as in \eqref{dd} (replacing $m_k$, $n_k$ by $m'_k$, $n_{k+d_k}$ respectively)
and define $r_{k+d_k+j}$ as \eqref{dx} for $1 \le j \le d'$ and $n_{k+d_k+j} = \lceil r_{k+d_k+j} \rceil$ for $1 \le j \le d'_k$. 
Note that $n_{k+d_k+d'_k} = m'_k$.  
We defined $n_{k+1}, n_{k+1}, \ldots, n_{k+d_k+d'_k}$ from $n_k$.    
Starting from $n_{k+d_k+d'_k}$ we repeat this procedure again.    
Inductively, we obtain a sequence $\{n_k\}$by this procedure satisfying 
$$ \limsup_k \frac{\varphi(n_k)}{\log n_k} = \gamma, \qquad  \liminf_k \frac{\varphi(n_k+1)}{\log (n_k+1)} = \delta.$$
On the other hand, we deduce from \eqref{pc} by letting $k$ to infinity that 
$$ \lim_{i \to \infty} \dfrac{\log r_{i+1} }{\log r_i} = C .$$   
Since $r_i \le n_i < r_i +1$ and $r_i$ goes to infinity, we get  
$$ \lim_{i \to \infty} \dfrac{\log n_{i+1} }{\log n_i} = C .$$    
Also by \eqref{pc} we have 
$$ \log n_i  \ge \log r_i \ge  C^{i-1} \log r_1 = C^{i-1} \log n_1.$$
It follows by the assumption of $C >1$ that  
$$\lim_{i \to \infty} \frac{\log n_i}{i} = \infty. $$   
%$$ \lim_{i \to \infty} \dfrac{\log n_{i+1} }{\log n_i} = C, \qquad \lim_{i \to \infty} \frac{\log n_i}{i} = \infty. $$

Now assume that $C=1$.
Let $n_1=3$. Then $1^2 \le \log n_1 < 2^2$.
For each integer $n_k$ with $k^2 \le \log n_k < (k+1)^2$, choose an integer $m_k$ such that
$$ \log m_k \ge (k+1)^2$$
and
\begin{align*}
\left| \frac{\varphi(m_k)}{\log m_k} - \gamma \right| < \frac 1k \text{ for } 0 \le \gamma < \infty, \qquad
%\frac{\varphi(m)}{\log m}  < - k, \ \text{ for } \gamma = - \infty, \qquad
\frac{\varphi(m_k)}{\log m_k}  > k, \text{ for } \gamma = \infty.
\end{align*}
Denote
\begin{equation}\label{dd2}
d_k = \left \lfloor  \sqrt{\log m_k}  - k \right\rfloor.
\end{equation}
For each $ 1 \le j < d_k$  put
\begin{equation*}
 n_{k + j} = \left \lceil \exp((k+j)^2) \right\rceil \text{ and } n_{k+d_k} = m_k.
\end{equation*}
Then for all $ 1 \le j \le d_k$
\begin{equation}\label{bn} (k+j)^2 \le \log n_{k+j} < (k+j+1)^2.
\end{equation}

Next, we choose $m_k'$ such that
$$ \log m'_k \ge ( k+d_k+1)^2 $$
and
\begin{align*}
\left| \frac{\varphi(m'_k+1)}{\log (m'_k+1)} - \delta \right| < \frac 1{k+d_k}, \quad &\text{ for } 0 \le \delta < \infty,\\
%\frac{\varphi(m'+1)}{\log (m'+1)}  < - (k+d), \quad &\text{ for } \delta = - \infty, \\
\frac{\varphi(m'_k+1)}{\log (m'_k+1)}  > k+d_k, \quad &\text{ for } \delta = \infty.
\end{align*}
Then, denote $d'_k = \left \lfloor  \sqrt{\log m'}  - (k+d_k) \right\rfloor$ as before
and define $n_{k+d_k+j} = \left \lceil \exp((k+d_k+j)^2) \right\rceil$ ($1 \le j < d'_k$) and $n_{k+d_k+d'_k} = m'_k$.
Repeating this procedure, we obtain a sequence $\{n_k\}$ satisfying
$$ \limsup_k \frac{\varphi(n_k)}{\log n_k} = \gamma, \qquad  \liminf_k \frac{\varphi(n_k+1)}{\log (n_k+1)} = \delta$$
and $$k^2 \le \log n_k < (k+1)^2,$$
which implies that
\begin{equation*}
\lim_{i \to \infty} \dfrac{\log n_{i+1} }{\log n_i} = 1, \qquad \lim_{i \to \infty} \frac{\log n_i}{i} = \infty. \qedhere
\end{equation*}
\end{proof}

\begin{lemma}\label{subseq2}
Let $\varphi(n)$ be a positive monotone increasing function which tends to infinity as $n \to \infty$. 

(i) We can choose a sequence of positive integers $\{m_i\}_{i =1}^\infty$ satisfying that  for each $i \ge 1$   
$$\varphi(m_{i+1}) - \varphi(m_i+1) \le 3, \qquad  \varphi(m_{i+1}) - \varphi(m_i) > 1. $$

(ii) We can choose a sequence of positive integers $\{m_i\}_{i =1}^\infty$ satisfying that    
$$\frac{\varphi(m_{i+1})}{\varphi(m_i+1)} \to 1 \text{ and } \frac{\log m_{i+1}}{\log m_i} \to 1$$
and for each $i$,  either $m_{i+1} \ge m_i \log m_i$ or $\varphi(m_{i+1}) - \varphi(m_i) > 1$.  
\end{lemma}

\begin{proof}
(i) Choose $n_1 \ge 3$. For each $n_i$ we choose $n_{i+1}$ as   
$$n_{i+1} =  \min \{n : \varphi(n) > \varphi(n_i) +1 \}. $$        
and define a sequence $\{m_i\}$ as    
$$m_i = \begin{cases}
n_i, & \text{ if } \varphi(n_{i+1}) - \varphi(n_i) \le 2, \\
n_{i+1}- 1, & \text{ if } \varphi(n_{i+1}) - \varphi(n_i) > 2. \\
\end{cases}$$
Then, we have
\begin{equation}\label{12.4}
n_i \le m_i \le n_{i+1}-1
\end{equation}
and
\begin{equation}\label{12.5}
\varphi(n_{i+1}-1) - \varphi(n_{i}) \le 1.
\end{equation}
If $\varphi(n_{i+1}) - \varphi(n_i) \le 2$, %note that $n_i = m_i$ and $m_{i+1}\leq n_{i+2}-1$, 
then we have from \eqref{12.4} and \eqref{12.5} 
\begin{equation*}\begin{split}
\varphi(m_{i+1}) -  \varphi(m_i+1) &\le \varphi(n_{i+2}-1) -  \varphi(n_i) \\
&=  \varphi(n_{i+2}-1) -  \varphi(n_{i+1}) + \varphi(n_{i+1}) -  \varphi(n_i) \\
&\le 1+ 2 = 3
\end{split}\end{equation*}
and
$$\varphi(m_{i+1}) - \varphi(m_i) \ge \varphi(n_{i+1}) - \varphi(n_i) > 1.$$
If $\varphi(n_{i+1}) - \varphi(n_i) > 2$, we have $m_i = n_{i+1} -1$ so that by  \eqref{12.5}
\begin{equation*}
\varphi(m_{i+1}) -  \varphi(m_i+1) \le \varphi(n_{i+2}-1) -  \varphi(n_{i+1}) \le 1
\end{equation*}
and
\begin{align*}
\varphi(m_{i+1}) - \varphi(m_{i}) &\ge \varphi(n_{i+1}) - \varphi(n_{i+1}-1) \\
&=  \varphi(n_{i+1}) -\varphi(n_i)+\varphi(n_i) - \varphi(n_{i+1}-1)> 1. 
\end{align*}

(ii) Choose $n_1 \ge 3$. For each $n_i$ we choose $n_{i+1}$ as   
$$n_{i+1} = \begin{cases} \min \{n : \varphi(n) > \varphi(n_i) +1 \}, &\text{if } \varphi(\lceil n_i\log n_i \rceil) > \varphi(n_i) + 1, \\    
\lceil n_i \log n_i \rceil, &\text{otherwise} \end{cases}$$        
and define a sequence $\{m_i\}$ as    
$$m_i = \begin{cases}
n_i, & \text{ if } \varphi(n_{i+1}) - \varphi(n_i) \le 2, \\
n_{i+1}- 1, & \text{ if } \varphi(n_{i+1}) - \varphi(n_i) > 2. \\
\end{cases}$$
Then, we have
\begin{equation}\label{2.4}
n_{i+1} < n_i \log n_i +1, \qquad n_i \le m_i \le n_{i+1}-1
\end{equation}
and
\begin{equation}\label{2.5}
\varphi(n_{i+1}-1) - \varphi(n_{i}) \le 1.
\end{equation}
If $\varphi(n_{i+1}) - \varphi(n_i) \le 2$, %note that $n_i = m_i$ and $m_{i+1}\leq n_{i+2}-1$, 
then we have from \eqref{2.4} and \eqref{2.5}    
\begin{equation*}\begin{split}
\varphi(m_{i+1}) -  \varphi(m_i+1) &\le \varphi(n_{i+2}-1) -  \varphi(n_i) \\
&=  \varphi(n_{i+2}-1) -  \varphi(n_{i+1}) + \varphi(n_{i+1}) -  \varphi(n_i) \\
&\le 1+ 2 = 3
\end{split}\end{equation*}
and if $\varphi(n_{i+1}) - \varphi(n_i) > 2$, we have $m_i = n_{i+1} -1$ so that by \eqref{2.5}
\begin{equation*}
\varphi(m_{i+1}) -  \varphi(m_i+1) \le \varphi(n_{i+2}-1) -  \varphi(n_{i+1}) \le 1.
\end{equation*}
Thus, we have $$\lim_{i \to \infty} \frac{\varphi(m_{i+1})}{\varphi(m_i+1)} = 1.$$
It follows from \eqref{2.4} that 
\begin{equation*}\begin{split}
m_{i+1} &\le n_{i+2} -1 < n_{i+1} \log n_{i+1} 
< (n_{i} \log n_{i} +1) \log (n_{i} \log n_{i} +1 ) \\
&\le (m_{i} \log m_{i} +1) \log (m_{i} \log m_{i} +1),
\end{split}\end{equation*}
thus
$$ \lim_{i\to \infty} \frac{\log m_{i+1}}{\log m_i} = 1.$$
If $\varphi(n_{i+1}) - \varphi(n_i) \le 2$ and $\varphi(\lceil n_i \log n_i \rceil) > \varphi(n_i) + 1$, then
$$\varphi(m_{i+1}) - \varphi(m_i) \ge \varphi(n_{i+1}) - \varphi(n_i) > 1.$$
If $\varphi(n_{i+1}) - \varphi(n_i) \le 2$ and $\varphi(\lceil n_i \log n_i \rceil) \le \varphi(n_i) + 1$, then
$$m_{i+1} \ge n_{i+1} = \lceil n_{i} \log n_i \rceil = \lceil m_{i} \log m_i \rceil \ge m_ i\log m_i.$$
If $\varphi(n_{i+1}) - \varphi(n_i) > 2$, then, by \eqref{2.5}, we have
\begin{align*}
\varphi(m_{i+1}) - \varphi(m_{i}) &\ge \varphi(n_{i+1}) - \varphi(n_{i+1}-1) \\
&=  \varphi(n_{i+1}) -\varphi(n_i)+\varphi(n_i) - \varphi(n_{i+1}-1)> 1. \qedhere 
\end{align*}
\end{proof}

Now we are ready to show the left part of Theorem~\ref{generalrate}.
%By the monotonicity of $\varphi(n)$, we have
%$$\limsup_{n \to \infty} \frac{\varphi(n)}{\log n} = \limsup_{i \to \infty} \frac{\varphi(m_i)}{\log m_i}, \quad
%\liminf_{n \to \infty} \frac{\varphi(n)}{\log n} = \liminf_{i \to \infty} \frac{\varphi(m_i)}{\log m_i}.$$

\begin{proof}[Proof of the one-dimensional part of Theorem~\ref{generalrate}]
Recall 
$$\limsup_{n \to \infty} \frac{\varphi(n)}{\log n} = \gamma, \quad
\liminf_{n \to \infty} \frac{\varphi(n)}{\log n} = \delta. $$

(i) Suppose that $\alpha = \beta = \infty$.

Choose $n_i = i$  and
 $\ell_{i} = \max \{  e^{i \varphi(i)}, \ell_{i-1} + n_{i-1} + 3 , i^2(i+3) \} $, $\ell_0 = 1$.
Then $\{ n_i \}$ and $\{ \ell_i \}$ and  satisfy the conditions of Lemma~\ref{Cantorsetlemma},  
thus for any $x \in A(\{n_i\}, \{ \ell_i \})$, note that $R_n(x)=\ell_i$ for all $ n_{i-1}+1\leq n\leq n_i$, we get
$$\lim_{n \to \infty} \frac{\log R_n (x)}{\varphi (n)} \geq \lim_{i\to\infty}\frac{\log\ell_i}{\varphi(n_i)} \ge \lim_{i \to \infty} \frac{ i \varphi(i) }{\varphi (i)} = \infty.$$ 
Therefore, $A(\{n_i\}, \{ \ell_i \})\subset E^{\varphi}_{\alpha,\beta}$.

(ii) Suppose that $\beta = \infty$, $0\le \alpha < \infty$ and $\gamma \in [1/ \alpha, \infty]$ for $\alpha >0$, $ \gamma = \infty$ for $\alpha  =0$.

We can choose an increasing sequence of positive integers $\{n_i\}$ such that
$$
\lim_{i \to \infty} \frac{\varphi(n_i)}{\log n_i} = \gamma,
$$
$$
\varphi(n_i) > i \varphi (n_{i-1}+1), \ 
\log n_i > i\varphi (n_{i-1}+1), \
\log n_{i} > \log n_{i-1} + i^2+2 
$$
and if $\gamma = \infty$, 
$$\frac{\varphi(n_{i+1})}{\log n_{i+1}} > \frac{\varphi(n_i)}{\log n_i}.$$

Let
$$\ell_i := 
\begin{cases}
\lceil e^{\alpha \varphi(n_i)} \rceil, &\text{ if } 0 < \alpha < \infty, \gamma = \infty, \\
\lceil (n_i)^{\alpha\gamma} \log n_i \rceil, &\text{ if } 0 < \alpha < \infty, \gamma < \infty \\
\lceil n_i \log n_i \rceil, &\text{ if } \alpha = 0, \gamma = \infty.
\end{cases}
$$
If $0 < \alpha < \infty, \gamma = \infty$, then for $i$ large enough to make $\frac{\varphi(n_{i+1})}{\log n_{i+1}} > \frac{\varphi(n_i)}{\log n_i} >  \frac 2{\alpha}$  and  $n_i \ge 3$,
\begin{align*}
\ell_{i+1}  - \ell_i  &\ge e^{\alpha \varphi(n_{i+1})}- e^{\alpha \varphi(n_i)} -1
= e^{\alpha \varphi(n_i)} \left( e^{\alpha (\varphi(n_{i+1}) - \varphi(n_i) )} -1 \right)  -1 \\
&> e^{2 \log n_i} \left( e^{\alpha \left (\frac{\log n_{i+1}}{\log n_i}\varphi(n_i) - \varphi(n_i)\right)} -1 \right)  -1 \\
&= (n_i)^2 \left( e^{\frac{\alpha\varphi(n_i)}{\log n_i} \left (\log n_{i+1}  - \log n_i \right)} -1 \right)  -1 \\
&> (n_i)^2 \left( e^{2 \left (\log n_{i+1}  - \log n_i \right)} -1 \right)  -1 \\
&= (n_{i+1})^2 - (n_i)^2 - 1  = (n_{i+1}- n_i)(n_{i+1} + n_i)  -1 \ge n_i +3 
\end{align*} 
and 
\begin{align*}
\lim_{i \to \infty} \frac{i(n_i+3)}{\ell_i} &\le \lim_{i \to \infty} \frac{i(n_i+3)}{n_i^2 } \\
&= \left( \lim_{i \to \infty} \frac{n_i+3}{n_i} \right) \cdot \left( \lim_{i \to \infty} \frac{i}{n_i} \right) 
= 1 \cdot 0 = 0. \end{align*}

For other cases of $0 < \alpha < \infty, \gamma < \infty$ and $\alpha = 0, \gamma = \infty$, for large $i$ with $n_i \ge 4$     
$$\ell_{i+1}  - \ell_i  \ge n_i (\log n_{i+1} - \log n_i) - 1 > 2n_i  -1 \ge n_i +3 $$ 
and 
\begin{align*}
\lim_{i \to \infty} \frac{i(n_i+3)}{\ell_i} &\le \lim_{i \to \infty} \frac{i(n_i+3)}{n_i \log n_i } \\
&= \left( \lim_{i \to \infty} \frac{n_i+3}{n_i} \right) \cdot \left( \lim_{i \to \infty} \frac{i}{\log n_i} \right) 
= 1 \cdot 0 = 0. \end{align*}

Therefore, $\{n_i\}$ and $\{ \ell_i \}$ satisfy the conditions of Lemma~\ref{Cantorsetlemma}.
Since 
$$
\lim_{i \to \infty} \frac{ \log\ell_i }{\varphi (n_{i})} 
=\begin{cases}
\alpha , &\text{ if } 0 < \alpha < \infty, \gamma = \infty, \\
\lim \frac{\log \ell_i}{\log n_i} \lim \frac{\log n_i}{\varphi(n_i)} = \alpha \gamma \cdot \frac{1}{\gamma} = \alpha , &\text{ if } 0 < \alpha < \infty, \gamma < \infty, \\
\lim \frac{\log \ell_i}{\log n_i} \lim \frac{\log n_i}{\varphi(n_i)} = 1  \cdot \frac{1}{\gamma} = 0, &\text{ if } \alpha = 0, \gamma = \infty,
\end{cases}
$$
for any $x \in A(\{n_i\}, \{ \ell_i \})$ we get
$$ 
\liminf_{n \to\infty} \frac{\log R_n (x)}{\varphi (n)} = \liminf_{i \to \infty} \frac{ \log\ell_i }{\varphi (n_{i})} 
= \alpha
$$
and
\begin{align*}
&\limsup_{n\to \infty} \frac{\log R_n (x)}{\varphi (n)} 
= \limsup_{i \to \infty} \frac{ \log \ell_i }{\varphi (n_{i-1}+1)} \\
&=  \begin{cases} 
\lim  \frac{ \log \ell_i }{\varphi (n_i)} \lim \frac{\varphi (n_i)}{\varphi (n_{i-1}+1)} = \alpha \cdot \infty = \infty , &\text{ if } 0 < \alpha < \infty, \gamma = \infty, \\
\lim \frac{\log \ell_i}{\log n_i} \lim \frac{\log n_i}{\varphi(n_{i-1}+1)} = \alpha \gamma \cdot \frac{1}{\gamma} = \alpha , &\text{ if } 0 < \alpha < \infty, \gamma < \infty, \\
\lim \frac{\log \ell_i}{\log n_i} \lim \frac{\log n_i}{\varphi(n_{i-1}+1)} = 1  \cdot \frac{1}{\gamma} = 0, &\text{ if } \alpha = 0, \gamma = \infty.
\end{cases}
%&= \limsup_{i \to \infty} \left( \frac{ \log \ell_i }{\varphi (n_i)} \frac{\varphi (n_i)}{\varphi (n_{i-1}+1)} \right) \\
%&= \left( \lim_{i \to \infty}  \frac{ \log \ell_i }{\varphi (n_i)} \right) \left( \lim_{i \to \infty} \frac{\varphi (n_i)}{\varphi (n_{i-1}+1)} \right) = \alpha \cdot \infty = \infty.
\end{align*}
Thus $A(\{n_i\}, \{ \ell_i \})\subset E^{\varphi}_{\alpha,\beta}$.

%We can choose $\ell_i = e^{\alpha \varphi(n)}$ and $e^{\beta \varphi(n)}$ for $0<\alpha\le \beta < \infty$.
%If $\alpha = 0$, then $\ell_i = e^{\varphi(n)/\sqrt{n}}$ and $e^{\beta \varphi(n)}$.

In the rest of the proof we only consider the case of  finite $\alpha, \beta$.  
Define 
$$
A := \begin{cases} \alpha \gamma, &\text{ if } 0< \alpha, \gamma < \infty, \\
1, &\text{ if } \alpha = 0, \gamma = \infty, \\
\infty, &\text{ if } \alpha > 0, \gamma = \infty, 
\end{cases}
\quad
B := \begin{cases} \beta \delta, &\text{ if } 0<  \beta, \delta < \infty, \\
1, &\text{ if } \beta = 0, \delta = \infty, \\
\infty, &\text{ if } \beta > 0, \delta = \infty.
\end{cases}
$$

(iii) Suppose that $0 \le \alpha \le \beta < \infty$ and $A = B = \infty$, i.e., $0 < \alpha \le \beta < \infty$ 
and $\delta = \gamma = \infty$. 

Let $\{ m_i\}$ be the sequence given by Lemma~\ref{subseq2} (i), thus
$\varphi (m_{i+1}) \le \varphi(m_i +1) + 3$ and $\varphi (m_{i+1}) > \varphi(m_i) + 1$.
%If $\alpha = \beta$, let $$\tilde \ell_i = e^{\alpha \varphi(m_i)}.$$

%Suppose that $\alpha < \beta$.
Let $\tilde \ell_1 = e^{\beta \varphi(m_1)}$.
For each $k$ with $\tilde \ell_{k} = e^{\beta \varphi(m_k)}$,
 we will choose $k'$ and define $\tilde \ell_{k+1}, \tilde \ell_{k+2}, \dots, \tilde \ell_{k'}$, where 
$\tilde \ell_{k'} = e^{\beta \varphi(m_{k'})}$.
Then by repeating the procedure inductively we define $\{ \tilde \ell_i\}$.
Let $\tilde \ell_{k} = e^{\beta \varphi(m_k)}$,
then there exists a minimal integer $d \ge 1$ such that 
$$\varphi(m_{k+d}) \ge \left( \frac{2\beta}{\alpha} - 1 \right) \varphi(m_k).$$ 
Define for $1 \le j < d$
$$
\tilde \ell_{k + j} = e^{ \frac{2\beta -\alpha}{2}\varphi(m_{k})  + \frac{\alpha}{2} \varphi(m_{k+j})  } > e^{\alpha \varphi(m_{k+j})},
$$ 
and
$$
\tilde \ell_{k + d} = e^{\alpha \varphi(m_{k+d})}, \qquad  
\tilde \ell_{k + d+1} =  e^{\beta \varphi(m_{k+d+1})}.
$$
Then continue this procedure for $k' = k+d+1$.

For $0 \le j < d-1$ %such that $\alpha \varphi(n_i) > 4 \log n_i$ and  $n_i = i \ge 3$
\begin{align*}
\tilde \ell_{k+j+1}  - \tilde \ell_{k+j}  
%&\ge e^{\beta \varphi(m_{k})  + \alpha(\varphi(m_{k+j+1}) - \varphi(m_{k}))/2 }  - e^{\beta \varphi(m_{k})  + \alpha(\varphi(m_{k+j}) - \varphi(m_{k}))/2 }  \\
&= \tilde \ell_{k+j}\left( e^{ \frac\alpha2 ( \varphi(m_{k+j+1})- \varphi(m_{k+j})) } -1 \right)  \\
&> \tilde \ell_{k+j} \frac{\alpha( \varphi(m_{k+j+1})- \varphi(m_{k+j})) }{2}  
> \frac{\alpha}{2} \tilde \ell_{k+j} 
\end{align*} 
and
\begin{align*}
\tilde\ell_{k+d}  - \tilde\ell_{k+d-1}
&= \tilde \ell_{k+d-1} \left( e^{\alpha \varphi(m_{k+d}) - \frac{2\beta -\alpha}{2}\varphi(m_{k})  - \frac{\alpha}{2} \varphi(m_{k+d-1})  } -1 \right)  \\
&> \tilde \ell_{k+d-1} \left( \alpha \varphi(m_{k+d}) - \frac{2\beta - \alpha}{2} \varphi(m_k) - \frac{\alpha}{2}\varphi(m_{k+d-1}) \right) \\ 
&\ge \tilde \ell_{k+d-1} \left( \frac{\alpha}{2} \varphi(m_{k+d}) - \frac{\alpha}{2}\varphi(m_{k+d-1}) \right) > \frac{\alpha}{2} \tilde \ell_{k+d-1},
\end{align*} 
\begin{align*}
\tilde \ell_{k+d+1}  - \tilde\ell_{k+d} &= \tilde \ell_{k+d} \left( e^{\beta \varphi(m_{k+d+1}) - \alpha\varphi(m_{k+d}) } -1 \right) \\
&> \tilde \ell_{k+d} \left( \beta \varphi(m_{k+d+1}) - \alpha\varphi(m_{k+d}) \right) > \beta \tilde \ell_{k+d} > \frac{\alpha}{2} \tilde \ell_{k+d}.
\end{align*} 
Since $\delta = \infty$, for large $i_0$ so as to $\varphi(m_i) >  \frac{3}{\alpha} \log m_i$ for $i \ge i_0$ we have
$$
\tilde \ell_{i+1} - \tilde \ell_i > \frac{\alpha}{2} \tilde \ell_i \ge\frac{\alpha}{2} e^{\alpha \varphi(m_i)}
> \frac{\alpha}{2} (m_i)^3,
$$
thus for $i\ge i_0$  
$$
\tilde \ell_{i} > \frac{\alpha}{2}  (m_i)^3  \ge \frac{\alpha}{2}  i^2 m_i .
$$
Thence, $\{ m_i\}$ and $\{ \ell_i = \lceil \tilde \ell_i \rceil \}$ satisfy the conditions of Lemma~\ref{Cantorsetlemma}.
Thus, for any $x \in A(\{m_i\}, \{ \ell_i \})$ we get
$$ 
\liminf_{n \to\infty} \frac{\log R_n (x)}{\varphi (n)} 
= \liminf_{i \to \infty} \frac{ \log\ell_i }{\varphi (m_{i})} 
= \alpha 
$$
and
\begin{align*}
\limsup_{n \to\infty} \frac{\log R_n (x)}{\varphi (n)} 
&= \limsup_{i \to \infty} \frac{ \log \ell_i }{\varphi (m_{i-1}+1)} 
= \limsup_{i \to \infty} \left( \frac{ \log \ell_i }{\varphi (m_{i})} \frac{\varphi (m_{i})}{\varphi (m_{i-1}+1)} \right) \\
&= \left( \limsup_{i \to \infty}  \frac{ \log \ell_i }{\varphi (m_i)} \right) \left( \lim_{i \to \infty} \frac{\varphi(m_i)}{\varphi (m_{i-1}+1)} \right) = \beta \cdot 1 = \beta.
\end{align*}
That is,  $A(\{m_i\}, \{ \ell_i \})\subset E^{\varphi}_{\alpha,\beta}$.

(iv) Suppose that $0 \le \alpha \le \beta < \infty$ and $1 \le A < B = \infty$.

Since $\delta = \infty$, we deduce that $\gamma = \infty$ and $A = 1$.
We can choose a sequence $\{ n_i \}$ satisfying that  
$$
\lim_{i \to \infty} \frac{\log n_i}{\varphi(n_{i-1}+1)} = \beta > 0.
$$
Note that combined with the assumption $\lim\limits_{i \to \infty} \frac{\varphi(n_i)}{\log n_i} = \infty$,
we get $\frac{\log n_{i+1}}{\log n_i}$ tends to infinity as $i\to\infty$, which implies that $\frac{i}{\log n_i}$ converges to 0. 
Let 
$$\ell_i = \lceil n_i \log n_i \rceil. $$
Then for large $i$ such that $\log \frac{n_{i+1}}{n_i} \ge 2$ and  $n_i \ge 4$
$$\ell_{i+1}  - \ell_i  \ge n_i (\log n_{i+1} - \log n_i) - 1 > 2n_i  -1 \ge n_i +3 $$ 
and
\begin{align*}
\lim_{i \to \infty} \frac{i(n_i+3)}{\ell_i} &\le \lim_{i \to \infty} \frac{i(n_i+3)}{n_i \log n_i } \\
&= \left( \lim_{i \to \infty} \frac{n_i+3}{n_i} \right) \cdot \left( \lim_{i \to \infty} \frac{i}{\log n_i} \right) 
= 1 \cdot 0 = 0. \end{align*}
Therefore,  $\{n_i\}$ and $\{ \ell_i \}$ satisfy the conditions of Lemma~\ref{Cantorsetlemma}.
Thus, for any $x \in A(\{n_i\}, \{ \ell_i \})$ we get
$$ 
\liminf_{n \to\infty} \frac{\log R_n (x)}{\varphi (n)} = \liminf_{i \to \infty} \frac{ \log\ell_i }{\varphi (n_{i})} 
= \left( \lim_{i \to \infty} \frac{ \log\ell_i }{\log n_i} \right) \left( \lim_{i \to \infty} \frac{\log n_i}{\varphi (n_{i})} \right) = 1 \cdot 0 = 0 
$$
and
\begin{align*}
\limsup_{n \to\infty} \frac{\log R_n (x)}{\varphi (n)} 
&= \limsup_{i \to \infty} \frac{ \log \ell_i }{\varphi (n_{i-1}+1)} 
= \limsup_{i \to \infty} \left( \frac{ \log \ell_i }{\log n_i} \frac{ \log n_i}{\varphi (n_{i-1}+1)} \right) \\
&= \left( \lim_{i \to \infty}  \frac{ \log \ell_i }{\log n_i} \right) \left( \lim_{i \to \infty} \frac{\log n_i}{\varphi (n_{i-1}+1)} \right) = 1 \cdot \beta = \beta.
\end{align*}
That is, $A(\{n_i\}, \{ \ell_i \})\subset E^{\varphi}_{\alpha,\beta}$.

(v) Suppose that $0 \le \alpha \le \beta < \infty$ and $1 \le A \le B < \infty$. 
By Lemma~\ref{subseq1} there exists a sequence $\{ n_i \}$ satisfying that  
$$
\limsup_{i \to \infty} \frac{\varphi(n_i)}{\log n_i} = \gamma, \quad
\liminf_{i \to \infty} \frac{\varphi(n_i+1)}{\log (n_i+1)} = \delta, \quad   
\lim_{i \to \infty} \frac{\log n_{i+1}}{\log n_i} = \frac{B}{A}.
$$
Let
$$\ell_i = \lceil (n_i)^{A} \log n_i \rceil.$$
Then $\{ n_i \}$ and $\{ \ell_i \}$ satisfies the conditions of Lemma~\ref{Cantorsetlemma}.   

If $x \in A(\{n_i\}, \{ \ell_i \})$, then $R_n (x) = \ell_i$ for each $n_{i-1} < n \le n_{i}$. Thus for $x \in A(\{n_i\}, \{ \ell_i \})$ we have for $n_{i-1} +1 \le n \le n_{i}$
$$
 \frac{\log \ell_i}{\varphi(n_{i})} \le \frac{\log R_{n}(x)}{\varphi(n)} \le \frac{\log \ell_i}{\varphi(n_{i-1}+1)},
$$
where the equalities holds for $n = n_i$ and $n = n_{i-1}+1$ respectively.
Therefore, we have
\begin{align*}
\limsup_{n \to \infty} \frac{\log R_{n}(x)}{\varphi(n)} &=
\limsup_{i \to \infty} \frac{\log \ell_i}{\varphi(n_{i-1}+1)} \\
&= \limsup_{i \to \infty} \left( \frac{\log \ell_i}{\log n_i} \cdot \frac{\log n_i}{\log (n_{i-1}+1)} \cdot \frac{\log (n_{i-1}+1)}{\varphi(n_{i-1}+1)} \right) \\
&= \lim_{i \to \infty} \frac{\log \ell_i}{\log n_i} \cdot \lim_{i \to \infty} \frac{\log n_i}{\log (n_{i-1}+1)} \cdot \varlimsup_{i \to \infty} \frac{\log (n_{i-1}+1)}{\varphi(n_{i-1}+1)} \\
&=  A \cdot \frac{B}{A} \cdot \frac 1\delta =\frac{B}{\delta} = \beta
\end{align*}
and
\begin{align*}
\liminf_{n \to \infty} \frac{\log R_{n}(x)}{\varphi(n)} &=
\liminf_{i \to \infty} \frac{\log \ell_i}{\varphi(n_i)}
= \liminf_{i \to \infty} \left( \frac{\log \ell_i}{\log n_i} \cdot \frac{\log n_i}{\varphi(n_i)} \right) \\
&= \lim_{i \to \infty} \frac{\log \ell_i}{\log n_i} \cdot \liminf_{i \to \infty} \frac{\log n_i}{\varphi(n_i)}
=  A  \cdot \frac 1\gamma = \alpha.
\end{align*}
Therefore, $x\in E_{\alpha,\beta}^\varphi$. 
It follows that $A(\{n_i\}, \{ \ell_i \}) \subset E_{\alpha,\beta}^\varphi$.
%which implies $\dim_{\rm H}E_{\alpha,\beta}^\varphi=1$, together with Lemma~\ref{Cantorsetlemma}.

(vi) Suppose that $0 \le \alpha \le \beta < \infty$ and $1 \le B < A  \le \infty$,  
i.e., $\alpha > 0$ and $0<\delta<\infty$.  
%Note that $\gamma > \delta$ since $\alpha \le \beta$.

We may assume that for all $n$
$$ \delta \le \frac{\varphi(n)}{\log n} \le \gamma.$$

%Let
%$$\rho(x) = \frac{\alpha \gamma ( x - \delta) - \beta \delta ( x - \gamma)}{\gamma -\delta}
%=  \frac{\alpha \gamma - \beta \delta}{\gamma -\delta} x + \frac{ (\beta - \alpha) \gamma \delta}{\gamma -\delta}. $$  

 Since $\beta \delta < \alpha \gamma$ and  $\alpha \le \beta$, we get $\gamma > \delta$.
Put
\begin{equation*}
C : = \begin{cases} \frac{\alpha \gamma - \beta \delta}{\gamma -\delta}, &\text{ if } \gamma < \infty, \\
\alpha &\text{ if } \gamma = \infty,
\end{cases}
\quad
D : = \begin{cases} \frac{ (\beta - \alpha) \gamma \delta}{\gamma -\delta}, &\text{ if } \gamma < \infty, \\
(\beta-\alpha)\delta &\text{ if } \gamma = \infty.
\end{cases}
\end{equation*}
and let
$$\rho(x) := C x +D. $$
Note that, $C >0$ and $D \ge 0$. 
Thus $1\leq \beta\delta\leq\rho(x)\leq\alpha\gamma$ for all $\delta \le x \le \gamma$.

Let $\{ m_i \}$ be the sequence given by Lemma~\ref{subseq2}  (ii),    
thus we have
\begin{multline}\label{deltai}
\rho\left( \frac{\varphi(m_{i+1})}{\log m_{i+1}} \right) \log m_{i+1} - \rho\left(\frac{\varphi(m_i)}{\log m_i} \right) \log m_i  \\
\ge C (\varphi(m_{i+1}) - \varphi(m_i) ) := \Delta_i.   
\end{multline}
Choose
$$\tilde \ell_i = {m_i}^{\rho(\varphi(m_i)/\log m_i)} \varphi(m_i) \log m_i \ \text{ and } \  \ell_i = \lfloor \tilde \ell_i \rfloor.     
$$
Then we get 
\begin{equation*}
\begin{split}
\tilde \ell_{i+1} - \tilde \ell_i 
&\ge m_i^{\rho(\varphi(m_i)/\log m_i)} \left( e^{\Delta_i}\varphi(m_{i+1}) \log m_{i+1} - \varphi(m_i) \log m_i \right) \\
&\ge m_i \varphi(m_i)  \left( e^{\Delta_i} \log m_{i+1} - \log m_i \right) \\ 
&\ge m_i \varphi(m_i) \big( (1 + \Delta_i ) \log m_{i+1} - \log m_i \big) \\
&= m_i \varphi(m_i) \left( \Delta_i \log m_{i+1} + (\log m_{i+1} - \log m_i)  \right),
\end{split}
\end{equation*}   
where the first and the second inequalities are respectively from \eqref{deltai} and $\rho(\varphi(m_i)/\log m_i)\geq 1$, and the third holds since $e^x \ge 1+x$. 
By Lemma~\ref{subseq2} (ii),
\begin{equation*}
\tilde \ell_{i+1} - \tilde \ell_i  \ge m_i \varphi(m_i) 
\min \left(  C \log m_{i+1}, \ \log \log m_i \right).
\end{equation*}      
Thus, for $i$ large enough that $m_i \ge 4$, $\varphi(m_i) \ge 2$, $ C \log m_{i+1} \ge 1$ and $\log \log m_i \ge 1$ we get  
$$ \ell_{i+1} - \ell_i > \tilde \ell_{i+1} - \tilde \ell_i - 1 \ge  2 m_i -1
\ge m_i + 3.  $$    

It follows from Lemma~\ref{subseq2} (ii) that  
\begin{equation*}
\varphi(m_{i+1}) \log m_{i+1} - \varphi(m_i) \log m_i \ge \min \left( \varphi(m_i) \log \log m_i , \  \log m_i \right),  
\end{equation*}   
%where the right had side is a monotone increasing function to infinity, 
thus  for large $i \ge i_0$ with $\log m_{i_0} > \varphi(m_1) \log \log m_{i_0}$ we have
\begin{align*}
\varphi(m_{i}) \log m_{i} &\ge \varphi(m_{i_0}) \log m_{i_0} + \sum_{j=i_0}^{i-1} \min \left( \varphi(m_1) \log \log m_i , \  \log m_i \right)  \\
&\ge \varphi(m_{i_0}) \log m_{i_0} + \varphi(m_1)  \sum_{j=i_0}^{i-1} \log \log m_j,
\end{align*}
which implies that
\begin{align*}
\frac{\tilde \ell_i}{i m_i} &= \frac{ {m_i}^{\rho(\varphi(m_i)/\log m_i)} \varphi(m_i) \log m_i}{i m_i}
\ge \frac{\varphi(m_i) \log m_i}{i} \\
&> \frac{\varphi(m_1)}{i} \sum_{j=i_0}^{i-1} \log \log m_j  
\ge \frac{\varphi(m_1)}{i} \sum_{j=i_0}^{i-1} \log \log j \\
&\ge \frac{\varphi(m_1)}{i} \int_{i_0}^{i-1} \log \log x \, \mathrm{d}x  
%&\ge \frac {\varphi(m_1)}{i} ( (i-1)\log \log(i-1) - \textrm{li}(i-1) ) 
\to \infty
\end{align*}
as $i \to \infty$.     
Hence, $\{m_i\}$ and  $\{ \ell_i \}$ satisfy the condition of Lemma~\ref{Cantorsetlemma}.

By the choice of $\ell_i$ we have   
\begin{equation*}
\rho\left(\frac{\varphi(m_i)}{\log m_i} \right) \frac{\log m_{i}}{\varphi(m_{i})} \le
\frac{\log \ell_i}{\varphi(m_{i})} 
\le \rho\left(\frac{\varphi(m_i)}{\log m_i} \right) \frac{\log m_{i}}{\varphi(m_{i})} + \frac{\log (\varphi(m_i) \log m_{i})}{\varphi(m_{i})}   
\end{equation*}
Using the fact that  
$$ \frac{\rho(x)}{x} = C + \frac {D}{x} %\frac{\alpha \gamma - \beta \delta}{\gamma -\delta} + \frac{(\beta -\alpha) \gamma\delta}{(\gamma -\delta)x}
$$
 is monotone decreasing for $\delta \le x \le \gamma$, we deduce that   
$$
\alpha = \frac{\rho(\gamma)}{\gamma} \le \rho\left(\frac{\varphi(m_{i})}{\log m_{i}} \right) \frac{\log m_{i}}{\varphi(m_{i})} \le \frac{\rho(\delta)}{\delta}= \beta.   
$$
It fowllows that  
\begin{equation*}
\alpha \le \frac{\log \ell_i}{\varphi(m_{i})}\le \beta+ \frac{\log \varphi(m_i) + \log \log m_{i}}{\varphi(m_{i})}   
\end{equation*}

If $x \in A(\{m_i\}, \{ \ell_i \})$, then for large $n$ satisfying $m_{i-1} < n \le m_{i}$, $R_n (x) = \ell_i$, thus   
%for $x \in A(\{\ell_i\}, \{m_i\})$ we have for $m_{i-1} < n \le m_{i}$   
$$
\frac{\log \ell_i}{\varphi(m_{i})} \le \frac{\log R_{n}(x)}{\varphi(n)} \le \frac{\log \ell_i}{\varphi(m_{i-1}+1)}
=  \frac{\log \ell_i}{\varphi(m_i)} \frac{\varphi(m_i)}{\varphi(m_{i-1}+1)},
$$   
which implies that  
$$
\alpha \le \frac{\log R_{n}(x)}{\varphi(n)}
\le \left( \beta+ \frac{\log \varphi(m_i) + \log \log m_{i}}{\varphi(m_{i})} \right) \frac{\varphi(m_i)}{\varphi(m_{i-1}+1)}.
$$   
It follows from the fact that $\varphi(n) \ge (\delta /2) \log n$ for sufficiently large $n$ and $\log x / x \to 0$ as $x \to \infty$,   
$$ 
\lim_{i \to \infty} \frac{\log \varphi(m_i)}{\varphi(m_i)} = 0, \quad
\lim_{i \to \infty} \frac{\log \log m_{i}}{\varphi(m_i)} = 0.
$$   
Therefore, using Lemma~\ref{subseq2}  (ii) we get     
$$
\limsup_{n \to \infty} \frac{\log R_{n}(x)}{\varphi(n)} 
\le \beta \cdot  \lim_{i \to \infty} \frac{\varphi(m_i)}{\varphi(m_{i-1}+1)} = \beta.
$$

%$$\alpha \le \frac{\log R_n(x)}{\varphi(n)} \le \beta .$$

Since the sequence $\{m_i\}$ satisfies Lemma \ref{subseq2} (ii)   
and $\varphi$ is increasing, 
\begin{align*}
\limsup_{n \to \infty} \frac{\varphi(n)}{\log n} 
&= \limsup_{i \to \infty} \left( \max_{m_{i-1} +1 \le n \le m_i} \frac{\varphi(n)}{\log n} \right) 
\le \limsup_{i \to \infty} \frac{\varphi(m_i)}{\log m_{i-1}}  \\
&= \limsup_{i \to \infty} \frac{\varphi(m_i)}{\log m_i} \frac{\log m_i}{\log m_{i-1}}
= \limsup_{i \to \infty} \frac{\varphi(m_i)}{\log m_i} \cdot \lim_{i \to \infty} \frac{\log m_i}{\log m_{i-1}} \\
&= \limsup_{i \to \infty} \frac{\varphi(m_i)}{\log m_i} \le \limsup_{n \to \infty} \frac{\varphi(n)}{\log n},
\end{align*}    
which implies that  
$$
\limsup_{i \to \infty} \frac{\varphi(m_i)}{\log m_i} = \limsup_{n \to \infty} \frac{\varphi(n)}{\log n}=  \gamma.
$$
Similarly, we get   
%$$
%\liminf_{n \to \infty} \frac{\varphi(n)}{\log n} 
%&= \liminf_{i \to \infty} \left( \min_{m_{i-1} +1 \le n \le m_i} \frac{\varphi(n)}{\log n} \right) 
%\ge \liminf_{i \to \infty} \frac{\varphi(m_{i-1}+1)}{\log m_{i}} % \\
%&= \liminf_{i \to \infty} \frac{\varphi(m_i)}{\log m_i} \cdot \lim_{i \to \infty} \frac{\log m_i}{\log m_{i-1}}
%= \liminf_{i \to \infty} \frac{\varphi(m_i)}{\log m_i} \ge \liminf_{n \to \infty} \frac{\varphi(n)}{\log n} 
%$$
$$
\liminf_{i \to \infty} \frac{\varphi(m_i)}{\log m_i} = \liminf_{n \to \infty} \frac{\varphi(n)}{\log n}=  \delta.   
$$

%\begin{align*}
%\limsup_{i \to \infty} \frac{\varphi(m_i)}{\log m_i} = \limsup_{n \to \infty} \frac{\varphi(n)}{\log n} = \gamma, \\
%\liminf_{i \to \infty} \frac{\varphi(m_i)}{\log m_i} = \liminf_{n \to \infty} \frac{\varphi(n)}{\log n} = \delta.
%\end{align*}
Choose a subsequence $\{m_{i_k}\}$ of $\{m_i\}$ such that $\lim\limits_{k\to\infty}\frac{\varphi(m_{i_k})}{\log m_{i_k}}=\gamma$.
By the continuity of $\rho$, then
$$\lim_{k\to\infty}\frac{\log R_{m_{i_k}} (x)}{\varphi(m_{i_k})} 
=\lim_{k\to\infty}\frac{\log \ell_{i_k}}{\log m_{i_k}}\frac{\log m_{i_k}}{\varphi(m_{i_k})}
=\frac{1}{\gamma}\lim_{k\to\infty}\rho\left(\frac{\varphi(m_{i_k})}{\log m_{i_k}}\right)
%=\frac{\alpha\gamma}{\gamma}
=\alpha.$$   
Hence, we get $\liminf\limits_{n \to \infty} \frac{\log R_n(x)}{\varphi(n)} = \alpha$ 
and $\limsup\limits_{n \to \infty} \frac{\log R_n(x)}{\varphi(n)} = \beta$ in a similar way.  
Thus $x\in E_{\alpha,\beta}^\varphi$.  
%for $x \in A(\{\ell_i\}, \{m_i\})$   
%$$\limsup_{n \to \infty} \frac{\log R_n(x)}{\varphi(n)} = \beta, \quad \liminf_{n \to \infty} \frac{\log R_n(x)}{\varphi(n)} = \alpha. $$  
We have established that  
$A(\{m_i\}, \{ \ell_i \}) \subset E_{\alpha,\beta}^\varphi$. 

Therefore, we deduce that  $$A(\{n_i\}, \{ \ell_i \}) \subset E_{\alpha,\beta}^\varphi$$
 for all cases, which implies that 
$$\dim_{\rm H}E_{\alpha,\beta}^\varphi=1 .$$ 
by  Lemma~\ref{Cantorsetlemma}.
\end{proof}

\end{document}